\documentclass[11pt]{amsart}

\usepackage{amsfonts,amsmath,amssymb,amsthm}
\usepackage[latin1]{inputenc}

\usepackage{graphicx}
\usepackage{wrapfig} 
\usepackage{float}

\usepackage{ifpdf}
\ifpdf
\fi

\newcommand{\mybeginproof}{\emph{Proof.} } 
\newcommand{\myendproof}{\hfill\qedsymbol\vspace*{1ex}}

\newcommand{\Nn}{\mathbb{N}}
\newcommand{\Zz}{\mathbb{Z}}

\renewcommand {\leq}{\leqslant}
\renewcommand {\geq}{\geqslant}

\newcommand{\val}{\mathop{\mathrm{val}}\nolimits} 

\renewcommand {\le}{\leqslant}
\renewcommand {\ge}{\geqslant}
\renewcommand {\leq}{\leqslant}
\renewcommand {\geq}{\geqslant}

{\theoremstyle{plain}
\newtheorem{theorem}{Theorem}    
\newtheorem{lemma}[theorem]{Lemma}       
\newtheorem{proposition}[theorem]{Proposition}      
\newtheorem{corollary}[theorem]{Corollary}      
}
{\theoremstyle{remark}

\newtheorem*{remark*}{Remark}  

}

\title{Waring's problem for polynomials\linebreak\  in two variables}

\author{Arnaud Bodin}
\email{Arnaud.Bodin@math.univ-lille1.fr}

\author{Mireille Car}
\email{Mireille.Car@univ-cezanne.fr}

\address{Laboratoire Paul Painlev\'e, Math\'ematiques, Universit\'e 
Lille 1, 59655 Villeneuve d'Ascq Cedex, France}

\address{Universit\'e Paul C\'ezanne, Facult\'e de Saint-J\'er\^ome, 
13397 Marseille Cedex, France}



\date{\today}

\begin{document}

\begin{abstract}
We prove that all polynomials in several variables can be decomposed as the sums of $k$th powers:
$P(x_1,\ldots,x_n) = Q_1(x_1,\ldots,x_n)^k+\cdots + Q_s(x_1,\ldots,x_n)^k$, 
provided that elements of the base field are themselves sums of $k$th powers.
We also give bounds for the number of terms $s$ and the degree of the $Q_i^k$.
We then improve these bounds in the case of two variables polynomials of large degree
to get a decomposition $P(x,y) = Q_1(x,y)^k+\cdots + Q_s(x,y)^k$
with $\deg Q_i^k \le \deg P + k^3$ and $s$ that depends on $k$ and $\ln (\deg P)$.
\end{abstract}

\maketitle


\section{Introduction}
\label{sec:intro}

For any domain $A$ and any integer $k\ge 2$, let $W(A,k)$ denote the 
subset of $A$ formed by all finite sums of $k$th powers $a^{k}$ with $a\in A$. 
Let $\underline{w}_A(k)$  denote the least integer $s$, if it exists, 
such that for every element $a\in W(A,k)$, the equation
$$a = a_{1}^{k} + \cdots + a_{s}^{k}$$ 
admits solutions $(a_{1},\ldots,a_{s}) \in A^{s}$. 

\bigskip

The case of polynomial rings $K[t]$ over a field $K$ is of particular interest (see \cite{Va1}, \cite{GaVa}).  
The similarity between the arithmetic of the ring $\Zz$ and the arithmetic
 of the polynomial rings in a single variable $F[t]$ over a finite field $F$ 
with  $q$ elements led to investigate a restricted variant of Waring's problem 
over $F[t]$, namely the strict Waring problem. For $P\in F[t]$, a representation
$$  P = Q_{1}^{k} + \cdots + Q_{s}^{k} \quad \text{ with } \deg Q_{i}^k < \deg P +k,$$
and $Q_{i} \in F[t]$ is a {\it strict representation}.

For the strict Waring problem, analog to 
the classical numbers $g_{\Nn}(k)$ and $G_{\Nn}(k)$ have been defined 
as follows. Let $g_{F[t]}(k)$ (resp. $G_{F[t]}(k)$) denote the least 
integer $s$, if it exists, such that every polynomial in
$W(F[t],k)$ (resp. every polynomial in $W(F[t],k)$ of sufficiently 
large degree) may be written as a sum satisfying the 
strict degree condition. 

General results about Waring's problem for the ring of polynomials 
over a finite field may be found in  \cite{Pa}, \cite{Va1}, \cite{Va3}, 
\cite{Va4}, \cite{LiWo} for the unrestricted problem and in \cite{We}, 
\cite{Ku}, \cite{EfHa}, \cite{Ca1}, \cite{GaVa} for the strict Waring problem. 
Gallardo's method introduced in \cite{Ga1} and performed in \cite{CG1} 
to deal with Waring's problem for cubes was generalized in \cite{Ca1} 
and \cite{GaVa}  where bounds for $g_{F[t]}(k)$ and $G_{F[t]}(k)$  
were established when  $q$ and $k$ satisfy some conditions. 

\bigskip

The goal of this paper is a study of Waring's problem for the ring 
$F[x,y]$ of polynomials in two variables over a field $F$. 
As for the one variable case, two variations of Waring's problem may be considered. 
The first one, is the unrestricted Waring's problem; the second one takes degree conditions in account.  

In Section \ref{sec:onetwo} we start by some relations between Waring's 
problem for polynomials in one variable and Waring's problem for polynomials 
in $n\geq2$ variables. In Section \ref{sec:vandermonde}, we prove that, 
provided all elements of the field  $F$ are sums of $k$th powers, 
there exists a positive integer $s$ (depending on $F$ and $k$) such that 
every polynomial $P\in F[x,y]$ may be written as a sum 
\begin{equation}
\label{eq:dag}
P = Q_{1}^{k} + \cdots + Q_{s}^{k},
\tag{\dag}
\end{equation}
where for $i = 1,\ldots, s$, $Q_{i}$ is a polynomial of 
$K[x,y]$ such that $\deg Q_{i} \leq  \deg P$. We then prove 
various improvements, the goal being to have in representations 
(\ref{eq:dag}) a decomposition with the following properties: 
the first priority is to have the lowest possible degree for 
the polynomials $Q_{i}$ and the second priority is a small number of terms. 
In Section \ref{sec:strict}, we prove that  (\ref{eq:dag}) is possible for polynomials 
of large degree
with  $\deg Q_{i}^{k} \leq \deg P + k^{3}$,  
the number $s$ of terms depending on $F$, $k$ and $\deg P$. 
To do that, in Section \ref{sec:root}, we introduce the notion of {approximate root}.

\bigskip

Let $F$ be a field such that: $F$ has more than $k$ elements, the characteristic of $F$ does not divide $k$ and
each element of $F$ can be written as a sum of $w_F(k)$ $k$th powers of elements of $F$.
We summarize in the tabular below the different bounds we get for a decomposition
of a polynomial $P(x,y)$ of degree $d$ as a sum $P = \sum_{i=1}^{s}{Q_i^k}$.

\renewcommand{\arraystretch}{1.3}
\begin{center}
\begin{tabular}{|c|c|c|}
\cline{2-3} \multicolumn{1}{c|}\       & $\deg Q_{i}^{k}$   & $s$           \\
\hline Corollary \ref{cor:vandermonde} &    $kd$        & $kw_{F}(k)$  \\
\hline Proposition \ref{prop:strict}   & $d + 2(k-1)^2$ & $\frac{1}{2}k(d+1)(d+2)w_{F}(k)$ \\
\hline Proposition \ref{prop:strict2}  &  $2d + 4k^2$   & $k^2(2k-1)w_F(k)$ \\
\hline Theorem \ref{th:sum}            &  $d+k^3$       & $2k^3\ln(\frac{d}{k}+1)\ln(2k) + 7k^4\ln(k) w_F(k)^2$ \\
\hline 
\end{tabular}
\end{center}

\medskip 

The two basic results are Corollary \ref{cor:vandermonde} that give a 
decomposition with very few terms of high degree 
and Proposition \ref{prop:strict} with many terms of low degree. Our 
first main result is Proposition \ref{prop:strict2},
that provides a decomposition with terms of medium degree, but the number 
of terms depends only on $k$ and not on the degree of $P$.
Then Theorem \ref{th:sum} decomposes $P$, of sufficiently large degree $d\ge 2k^4$,
into a sum of few terms of low degree.

For instance, let a field with $w_{F}(k)=1$ (that is to say each element of $F$ is a $k$th power),
set $d=200$ and $k=3$, then each polynomial $P(x,y)$ of degree $200$ can be written $P = \sum_{i=1}^{s}{Q_i^3}$
with\footnote{In fact the last bound comes from a sharper bound obtained in the proof of Theorem~\ref{th:sum}.}  
\renewcommand{\arraystretch}{1.2}
\begin{center}
\begin{tabular}{|c|c|c|}
\cline{2-3} \multicolumn{1}{c|}\       & $\deg Q_i^k$ & $s$           \\
\hline Corollary \ref{cor:vandermonde} & $600$        & $3$  \\
\hline Proposition \ref{prop:strict}   & $208$        & $60903$ \\
\hline Proposition \ref{prop:strict2}  & $436$        & $45$ \\
\hline Theorem \ref{th:sum}            & $227$        & $812$ \\
\hline 
\end{tabular}
\end{center}

\section{The unrestricted Waring's problem}
\label{sec:onetwo}

If $A$ is a domain, we denote by $W(A,k,s)$ the set of elements $a\in A$ 
that can be written as a sum $a = a_{1}^{k} + \cdots + a_{s}^{k}$ with $a_{1}, \ldots, a_{s}\in A$; 
if $A = W(A,k,s)$ for an integer $s$, then for any integer $s' \geq s$, 
we have $A = W(A,k,s')$. Let $w_{A}(k)$ denote the least integer $s$ 
such that $A = W(A,k,s)$. If such a $s$ does not exist, let $w_{A}(k) = \infty$. 
Observe that $w_{A}(k) \ge \underline{w}_{A}(k)$ and in the case that $A=W(A,k)$ then
$w_{A}(k) = \underline{w}_{A}(k)$.
In this section we are concerned with rings of polynomials in $n\geq1$ variables. 

\begin{lemma} Let $A$ be a domain and let $s$ be a positive integer. 
\begin{enumerate}
  \item If $A[t] = W(A[t],k,s)$, then $A=W(A,k,s)$, so that $w_A(k) \le w_{A[t]}(k)$. 
  \item $A[t] = W(A[t],k,s)$ if and only if $A[x_1,\ldots, x_n] = W(A[x_1,\ldots, x_n],k,s)$,
 so that $w_{A[x_1,\ldots, x_n]}(k) = w_{A[t]}(k)$.  
\end{enumerate}
\end{lemma}

A kind of reciprocal to (1) will be discussed later in Proposition \ref{prop:vandermonde}.

{\it Proof} 
\begin{enumerate}
  \item Suppose $A[t] = W(A[t],k,s)$. Every $a \in A$ is a sum
$a = Q_{1}^{k}+\cdots + Q_{s}^{k}$ for some $Q_{i} \in A[t]$. Specializing $t$
at $1$ for instance, gives $a= Q_{1}(1)^{k}+\cdots + Q_{s}(1)^{k}$, 
a sum in $A$. Therefore, $w_{A[t]}(k) \geq w_{A}(k)$.

  \item 
  \begin{enumerate}
    \item If   $A[t] = W(A[t],k,s)$, then there exist $Q_{1},\ldots, Q_{s} \in A[t]$ 
such that $t = Q_{1}(t)^{k}+ \cdots + Q_{s}(t)^{k}$.
Pick $P \in A[x_1,\ldots, x_n]$ and substitute $P$ for $t$, we get:
$P(x_1,\ldots, x_n) = Q_{1}(P(x_1,\ldots, x_n))^k+\cdots+Q_{s}(P(x_1,\ldots, x_n))^k$.
Hence $w_{A[x_1,\ldots, x_n]}(k)\leq w_{A[t]}(k)$.
    \item If $A[x_1,\ldots, x_n] = W(A[x_1,\ldots, x_n],k,s)$ then any $P(t) \in A[t]$ 
can be written $P(t) = Q_{1}(t,x_2,\ldots, x_n)^k+\cdots+Q_{s}(t,x_2,\ldots, x_n)^k$.
By the specialization $x_2 = \cdots = x_n = 1$ we get that $P(t)\in W(A[t],k,s)$. 
Therefore $w_{A[x_1,\ldots, x_n]}(k)\geq w_{A[t]}(k)$.
  \end{enumerate}
\end{enumerate}

\begin{remark*}
It is also true that $A[t] = W(A[t],k,s)$ if and only if $t \in W(A[t],k,s)$.
\end{remark*}

This remark motivates the fact that we consider Waring's problem for a 
polynomial ring $F[x_1,\ldots, x_n]$ where $F$ is a field satisfying 
the condition $F = W(F,k)$. Such a field is called a 
{\it Waring field for the exponent} $k$, or briefly, a $k$-\emph{Waring field}. 

Let us give some examples. An algebraically closed field $F$ is a 
$k$-Waring field with $w_{F}(k) = 1$ for every positive integer $k$. 
If $F$ is a finite field of characteristic $p$, for every positive 
integer $n$, $F$ is a $p^{n}$-Waring field with $w_{F}(p^{n})=1$. 
It is known, c.f. \cite{Bh}, \cite{EfHa}, that  for a finite field 
$F$ of characteristic $p$ that does not divide $k$ and order $q=p^{m}$, 
$F$ is a Waring field for the exponent $k$ if and only if for all $d\neq m$ 
dividing $m$,  $(q-1)/(p^d -1)$ does not divide $k$.  

When $F$ has prime characteristic $p$, it is sufficient to consider 
Waring's problem for exponents $k$ coprime with $p$. Indeed, we have
\begin{proposition} 
Let $k\geq 2$  be coprime with $p$. Then, for any positive integer 
$\nu$ and for any positive integer $s$, we have
$$ W(F[x_1,\ldots, x_n],kp^{\nu},s) = \big\{Q^{p^{\nu}} \mid Q \in W(F[x_1,\ldots, x_n],k,s))\big\},$$
$$w_{F[x_1,\ldots, x_n]}(kp^{\nu}) = w_{F[x_1,\ldots, x_n]}(k).$$
\end{proposition} 

The proof is similar to that of \cite[Theorem 2.1]{Ca1}
and relies on the relation $(Q_1^k+\cdots+Q_s^k)^p = Q_1^{pk}+\cdots+Q_s^{pk}$.

\section{Vandermonde determinants}
\label{sec:vandermonde}

\subsection{Sum with high degree}

Let us recall that for $(\alpha_{1},\ldots,\alpha_{n})\in L^{n}$, 
where $L$ is a field containing $F$, Vandermonde's determinant
$V(\alpha_{1},\ldots,\alpha_{n})$ verifies:
\begin{equation}
\label{eq:vandermonde}
V(\alpha_1,\ldots,\alpha_n) := 
\begin{vmatrix}
1      & \alpha_1 & \alpha_1^2 & \cdots & \alpha_1^{n-1} \\  
1      & \alpha_2 & \alpha_2^2 & \cdots & \alpha_2^{n-1} \\ 
\vdots &          &            &        & \vdots         \\
1      & \alpha_n & \alpha_n^2 & \cdots & \alpha_n^{n-1} \\    
\end{vmatrix}
= \prod_{1\le i < j \le n}(\alpha_i-\alpha_j).
\end{equation}

\begin{proposition}
\label{prop:vandermonde}
Let $F$ be a field with more than $k$ elements, whose characteristic does not divide $k$, such that each element of $F$ can
be written as a sum of $k$th powers of elements of $F$.
Then any polynomial $P(x_1,\ldots,x_n)$ with coefficients in the field $F$ is a sum of $k$th powers.
In other words, for any positive integer $n$,
 $$F[x_1,\ldots,x_n] = W(F[x_1,\ldots,x_n],k).$$
\end{proposition}

\begin{proof}
The proof follows ideas from \cite{GaVa}.
Let $\alpha_1,\ldots,\alpha_k$ be distinct elements of $F$.
First notice that by formula (\ref{eq:vandermonde}), if $t$ is any transcendental element over $F$,
$V(\alpha_1,\ldots,\alpha_k)=V(t+\alpha_1,\ldots,t+\alpha_k)$.
By expanding the determinant $V(t+\alpha_1,\ldots,t+\alpha_k)$ along the last column we get
(a term marked $\check{x_i}$ means that it is omitted):
\begin{align*}
V(\alpha_1,\ldots,\alpha_k)  &= V(t+\alpha_1,\ldots,t+\alpha_k) \\
   &= \pm \sum_{i=1}^k (-1)^i(t+\alpha_i)^{k-1} V(t+\alpha_1,\ldots,{\overbrace{t+ \alpha_i}^\vee},\ldots,t+\alpha_k) \\
   &= \pm \sum_{i=1}^k (-1)^i(t+\alpha_i)^{k-1} V(\alpha_1,\ldots,\check{\alpha_i},\ldots,\alpha_k). \\
\end{align*}
The constant $\gamma = V(\alpha_1,\ldots,\alpha_k)$ is non-zero 
since the $\alpha_i$ are distinct elements of $F$.
We write
$$
\sum_{i=1}^k \frac{(t+\alpha_i)^{k-1}}{\beta_i} = \gamma,
$$
where $\beta_i$ are non-zero constants in $F$.
This formula proves that the function 
$C(t) = \sum_{i=1}^k \frac{(t+\alpha_i)^k}{\beta_i}-\gamma k t$
has an identically null derivative; since the characteristic of $F$ does not divide $k$,
it implies that $C(t)$ is a constant. So that, for some $\delta \in F$:
\begin{equation}
\label{eq:eqsumk}
\sum_{i=1}^k \frac{(t+\alpha_i)^k}{\beta_i} = \gamma k t + \delta.
\end{equation}

Let $P(x_1,\ldots, x_n)\in F[x_1,\ldots, x_n]$. By substitution of  
$t$ by $(P-\delta)/(\gamma k)$ in equality (\ref{eq:eqsumk}) we get
$P = \sum_{i=1}^{k}\frac{(P-\delta+\alpha_{i}\gamma k)^{k}}{\beta_{i}(\gamma k)^{k}}$.
But by assumption $1/\beta_{i}(\gamma k)^{k}$ is a sum of $k$th powers of elements of $F$.
So that $P(x_1,\ldots, x_n)$ is also a sum of $k$th powers of elements of $F[x_1,\ldots, x_n]$.
\end{proof}

\begin{corollary}
\label{cor:vandermonde}
Let $F$ have more than $k$ distinct elements such that its characteristic does not divide $k$. 
Every polynomial
 $P(x_1,\ldots, x_n)\in F[x_1,\ldots, x_n]$ of degree $d$ can be written as a sum
$$P(x_1,\ldots, x_n) = \delta_{1} Q_{1}(x_1,\ldots, x_n)^{k} +\cdots + \delta_k Q_{k}(x_1,\ldots, x_n)^{k},$$
where $\delta_1,\ldots,\delta_k \in F$ and  $Q_1,\ldots,Q_k$ are polynomials 
in $F[x_1,\ldots,x_n]$ such that $\deg Q_{i}^{k} \le kd$. 
If moreover each element of $F$ is a sum of $w_{F}(k)$ $k$th powers, then
$$P(x_1,\ldots, x_n) = Q_{1}(x_1,\ldots, x_n)^{k}+\cdots + Q_s(x_1,\ldots, x_n)^{k}$$
where $Q_1,\ldots,Q_s \in F[x_1,\ldots, x_n]$ such that $\deg {Q_{i}}^{k} \le kd$
for some $s \le k \cdot w_{F}(k)$.
\end{corollary}

\begin{proof}
It comes from formula (\ref{eq:eqsumk}) and the discussion below it.
\end{proof}

\medskip

\emph{In the sequel, we consider polynomials in two variables.}

\subsection{Low degree, many terms}
\begin{proposition}
\label{prop:strict}
Let $F$ be a field with more than $k$ distinct elements such that its characteristic does not divide $k$.
Every polynomial $P \in F[x,y]$ of degree $d$ admits a decomposition:
$$P(x,y) = \delta_1 Q_1(x,y)^k+\cdots+\delta_s Q_s(x,y)^k,$$
where $\delta_1,\ldots,\delta_s \in F$ and  $Q_1,\ldots,Q_s$ are polynomials 
in $F[x,y]$ such that $\deg Q_i^k \le d + 2(k-1)^2$ and $s \le k\cdot \frac{(d+1)(d+2)}{2}$.

If moreover each element of $F$ is a sum of $k$th powers then $P$ admits a decomposition:
$$P(x,y) = Q_1(x,y)^k+\cdots+Q_s(x,y)^k,$$
where $Q_1,\ldots,Q_s \in F[x,y]$ 
with $\deg Q_i^k \le d + 2(k-1)^2$ and $s \le k w_F(k)\frac{(d+1)(d+2)}{2}$.
\end{proposition}

\begin{proof}
Let $P(x,y) = \sum a_{i,j} x^iy^j$. We make the Euclidean divisions:
$i=pk+a$ and $j=qk+b$ with $0\le a,b < k$. Each monomial $x^iy^j$
can now be written $x^iy^j = (x^py^q)^k \cdot x^ay^b$.
By Corollary \ref{cor:vandermonde}, $x^a y^b$ can be written
$x^ay^b = \delta_1 Q_1(x,y)^k+\cdots+\delta_k Q_k(x,y)^k$
with $\delta_1,\ldots,\delta_k \in F$, $Q_1,\ldots,Q_k \in F[x,y]$
and $\deg Q_i \le \deg(x^ay^b)$,
so that 
$$x^iy^j  = \delta_1 (x^py^qQ_1(x,y))^k + \cdots + \delta_k (x^py^qQ_1(x,y))^k.$$
Moreover $\deg ((x^py^qQ_i(x,y))^k) = k(p+q+\deg Q_i) \le kp + kq +ka+kb
= i+j + (k-1)(a+b) \le i+j + 2(k-1)^2 \le d + 2(k-1)^2$.

As $\deg P = d$ the number of monomials $x^iy^j$ is less or equal than $\frac{(d+1)(d+2)}{2}$,
so that $P$ admits a decomposition $P(x,y) = \delta_1 Q_1(x,y)^k+\cdots+\delta_s Q_s(x,y)^k$
with $\deg Q_i^k \le d + 2(k-1)^2$ and $s \le k\frac{(d+1)(d+2)}{2}$.
Thus we can find a decomposition $P(x,y) = Q_1(x,y)^k+\cdots+Q_s(x,y)^k$
for some $s \le k w_F(k) \frac{(d+1)(d+2)}{2}$.
\end{proof}

\subsection{Medium degree, few terms}

We improve this method to get fewer terms in the sum but the degree of each term is higher.
\begin{proposition}
\label{prop:strict2}
Let $F$ be a field with more than $k$ elements, such that its characteristic does 
not divide $k$ and each element of $F$ is a sum of $k$th powers.
Any $P \in F[x,y]$ $P$ admits a decomposition:
$$P(x,y) = Q_1(x,y)^k+\cdots+Q_s(x,y)^k,$$
where  $Q_1,\ldots,Q_s$ are polynomials in $F[x,y]$
with $\deg Q_i^k \le 2\deg P + 4k^2$ and $s \le k^2(2k-1)w_F(k)$.
\end{proposition}
Observe that the bound for $s$ does not depend on the degree of the polynomial $P$.

\mybeginproof
\begin{wrapfigure}{}{0pt}
  \includegraphics[scale=1]{fig_waring.2}
\end{wrapfigure}
Let $d$ be the least multiple of $2k^2$ such that $d \ge \deg P$.
The Newton polygon of $P$ is included in the triangle $ABC$ with
$A(0,0)$, $B(0,d)$, $C(d,0)$.

We cover this triangle $ABC$ by $k(2k-1)$ small
triangles that are translations (by $\frac{d}{2k}$-units) of $A'B'C'$ with $A'(0,0)$, $B'(0,\frac dk)$, $C'(\frac dk,0)$.
This covering means that we can write $P(x,y)$ as a sum of $k(2k-1)$ polynomials of the form
$x^{i\frac{d}{2k}}y^{j\frac{d}{2k}}P_{i,j}(x,y)$ with $\deg P_{i,j} \le \frac dk$ and $0\le i+j \le 2k-2$
(so that $\deg x^{i\frac{d}{2k}}y^{j\frac{d}{2k}} < d$).
As $2k^2$ divides $d$ then $x^{i\frac{d}{2k}}y^{j\frac{d}{2k}}$ is a $k$th power.
Furthermore, by Corollary \ref{cor:vandermonde}, we can write each $P_{i,j}$ as a sum
of $k w_F(k)$ powers, each power being of degree at most $k \frac dk = d$.
Hence we get a decomposition $P(x,y)=Q_1(x,y)^k+\cdots+Q_s(x,y)^k$
with $s\le k^2(2k-1)w_F(k)$ terms and $\deg Q_i^k < 2d$.
\myendproof

\section{Approximate root}
\label{sec:root}

\begin{wrapfigure}{r}{0pt}
  \includegraphics[scale=1]{fig_waring.3}
\end{wrapfigure}
In this section $F$ is a field whose characteristic does not divide $k$.
Let $P\in F[x,y]$ be a polynomial that verifies the following conditions:
$\deg P \le d$,  $\deg_x P < m$.
So that the Newton polygon $\Gamma(P)$ of $P$ is
(included in) the following polygon $\bar \Gamma(P)$ (whose vertices are
$(0,0)$, $(m,0)$, $(m,n)$, $(0,d)$).
We set $n = d-m$ and we suppose that
$k | m, \quad k |n, \quad k | d.$
We will look for a $Q \in F[x,y]$ such that $\deg Q \le \frac d k,  \deg_x Q \le \frac m k$,
so that $\Gamma(Q^k) \subset \bar \Gamma(P)$. In fact the Newton polygon of $Q$ is homothetic to the one
of $P$ with a ratio $\frac 1k$.

\begin{proposition}
\label{prop:root}
There exists a unique $Q(x,y) \in F[x,y]$, monic in $x$, such that $P+x^my^n - Q^k$
has no monomial $x^iy^j$ with $i \ge m-\frac mk$ and $j \ge n - \frac nk$.
That is to say, the Newton polygon of $P + x^my^n -Q^k$ is (included in):
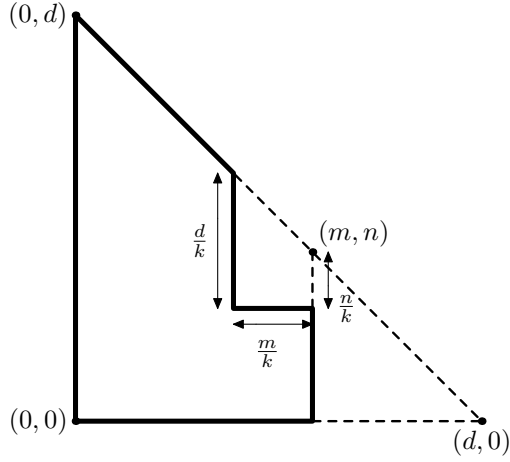
\begin{center}
\begin{figure}[H]
  \includegraphics[scale=1]{fig_waring.4}
\end{figure}
\end{center}
\end{proposition}

It means that with two $k$th powers ($x^my^n$ and $Q^k$) we ``cancel'' the trapezium $T$
(defined by the vertices $(m,n)$, $(m,n-\frac{n}{k})$, $(m-\frac{m}{k},n-\frac{n}{k})$,$(m-\frac{m}{k},n + \frac{d}{k} -\frac{n}{k})$).
This procedure is similar to the computation of the approximate $k$th root of a one variable polynomial, see \cite{Bo}.
The proof is sketched into the following picture:
{\center
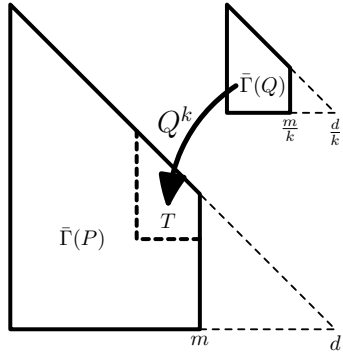
\begin{figure}[H]
  \includegraphics[scale=0.8]{fig_waring.5}
\end{figure}
}
Morally, the coefficients of $Q$ provide a set of unknowns, which is chosen in order that
$Q^k$ and $P$ can be identified into the trapezium area $(T)$.

\begin{proof}
We write $P$ as the sum $P=P_1+P_2$ corresponding to the decomposition into two areas
of $\bar \Gamma(P) = T \cup (\bar\Gamma(P)\setminus T)$:
we write $P_1$ as a polynomial in $x$ whose coefficients are in $F[y]$ so that
$P_1(x,y) =  a_1(y) x^{m-1}+\cdots + a_{\frac{m}{k}}(y) x^{m-\frac{m}{k}}$
with $\deg a_i(y) \le n+i$ and $\val a_i(y) \ge n - \frac nk$.
We denote by $\val$ the $y$-adic valuation: 
$\val \sum \alpha_i y^i = \min \{ i  \mid \alpha_i \neq 0\}$.

We set $P'_1(x,y) = y^n x^m+ P_1(x,y)$ and $a_0(y) = y^n$.
Notice that we have added a $k$th power since $k|m$ and $k|n$.

We also write $Q(x,y)$ as a polynomial in $x$ with coefficients in $F[y]$:
$Q(x,y) = b_0(y) x^{\frac mk} + b_1(y) x^{\frac mk -1} + \cdots + b_{\frac{m}{k}}(y)$.

We now identify the monomials of $P'_1(x,y) = x^my^n + P_1(x,y)$ with the monomials of $Q(x,y)^k$, in the trapezium $T$.
As we only want to identify the monomials of a sufficiently high degree we define the following equivalence:
$$a(y) \bumpeq b(y) \quad \text{if and only if} \quad \deg (a(y) - b(y)) < n - \frac nk.$$
It yields the following polynomial system of equations ($a_i(y)$ are data, and $b_i(y)$ unknowns):
{\small
\begin{equation*}
\label{eq:sys}\tag{$\mathcal{S}$}
\begin{cases}
  a_0 \bumpeq b_0^k \\
  a_1 \bumpeq k b_0^{k-1}b_1 \\
  a_2 \bumpeq k b_0^{k-1}b_2+ \binom{k}{2} b_0^{k-2}b_1^2 \\
  \vdots \\
  a_\ell \bumpeq k b_0^{k-1}b_\ell + {\hspace*{-1.5em} 
\displaystyle{\sum_{\substack{i_1+2i_2+\cdots+(\ell-1)i_{\ell-1}=\ell \\ i_0+i_1+i_2+\cdots + i_{\ell-1}=k}}
\hspace*{-1.5em} c_{i_1\ldots i_{\ell-1}} b_0^{i_0}b_1^{i_1}\cdots b_{\ell-1}^{i_{\ell-1}}}}, 
\qquad 1 \le \ell \le \frac m k, \\
\end{cases}
\end{equation*}
}
\hspace{-2ex}  
where the coefficients $c_{i_1\ldots i_{\ell-1}}$ are the multinomial coefficients defined by the following formula:
$$c_{i_1\ldots i_{\ell-1}} = \binom{k}{i_1,\ldots,i_{\ell-1}} = \frac{k!}{i_1! \cdots i_{\ell-1}!(k-i_1-\cdots-i_{\ell-1})!}.$$
The first equation has a solution $b_0(y) = y^{\frac nk}$.
Then, as $\val a_1(y) \ge n - \frac nk$, we have $b_1(y) = \frac 1k \frac{a_1(y)}{b_0(y)^{k-1}} \in F[y]$ ($k$ is invertible in $F$).
Next we compute $b_2(y)$,... by induction using the fact that system (\ref{eq:sys}) is triangular.
Suppose that $b_0(y),b_1(y),\ldots,b_{\ell-1}(y)$ have been found.
System (\ref{eq:sys}) provides the relation:
$$a_\ell \bumpeq k b_0^{k-1}b_\ell + \sum c_{i_1\ldots i_{\ell-1}} b_0^{i_0}b_1^{i_1}\cdots b_{\ell-1}^{i_{\ell-1}}.$$
As $b_0(y) = y^{\frac nk}$ it means that the polynomials
$ky^{n-\frac nk} b_\ell(y)$ and $a_\ell- \sum c_{i_1\ldots i_{\ell-1}} b_0^{i_0}b_1^{i_1}\cdots b_{\ell-1}^{i_{\ell-1}}$
have equal coefficients associated to monomials $y^i$ with $i\ge n-\frac nk$.
Whence $b_\ell(y)$ is uniquely determined.
We have proved that system (\ref{eq:sys}) has a unique solution $(b_0(y),b_1(y),\ldots,b_{\frac mk}(y))$.

Finally, we need to prove that $\deg b_i \le \frac nk + i$ for $0 \le i \le  \frac mk$.
We have $b_0(y) = y^{\frac nk}$, so that $\deg b_0 = \frac nk$
and $b_1(y) = \frac 1k \frac{a_1(y)}{\left(y^{\frac nk}\right)^{k-1}}$; 
thus, $\deg b_1 \le \deg a_1 - n + \frac nk \le n+1 -n +\frac nk = \frac nk +1$.
Then, by induction we get
\begin{align*}
  \deg b_0^{i_0}b_1^{i_1}\cdots b_{\ell-1}^{i_{\ell-1}} 
       &\le   i_0 \left(\frac nk+0\right)+i_1\left(\frac nk+1\right)+\cdots + i_\ell\left(\frac nk+\ell\right) \\
       &=     \frac nk (i_0+i_1+\cdots+i_\ell) + i_1+2i_2+\cdots+(\ell-1)i_{\ell-1} \\
       &=     \frac nk k + \ell \\
       &=     n + \ell .
\end{align*}
We also find  $\deg a_\ell \le n +\ell$ so that $\deg b_\ell \le \frac nk + \ell$.
\end{proof}

\section{Strict sum of $k$th powers}
\label{sec:strict}

This section is devoted to the proof of the main theorem:
\begin{theorem}
\label{th:sum}
Let $F$ be a field with more than $k$ elements, whose characteristic does not divide $k$, 
such that each element of $F$ can be written as a sum
of $w_F(k)$ $k$th powers of elements of $F$.
Each polynomial  $P(x,y) \in F[x,y]$ of degree  $d \ge 2k^4$ is the sum of $k$th powers
$$P(x,y) = Q_1(x,y)^k+\cdots+Q_s(x,y)^k,$$
of polynomials $Q_i\in F[x,y]$ with $\deg Q_i^k \le d + k^3$ and 
$s \le 2k^3\ln(\frac{d}{k}+1)\ln(2k) + 7k^4\ln(k) w_F(k)^2$.
\end{theorem}

The bound for $s$ is derived from a sharper bound given at the end of the proof.
We start by sketching the proof by pictures:
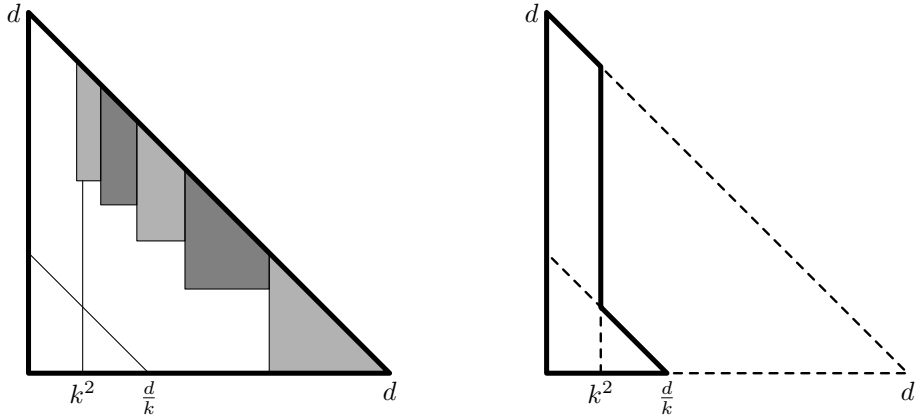
\begin{center}
\begin{figure}[H]
  \includegraphics[scale=1]{fig_waring.6}
\qquad\qquad
  \includegraphics[scale=1]{fig_waring.8}
\end{figure}
\end{center}

We consider the Newton polygon of $P$, it is included in a large triangle
(see the left figure).
We first cut off trapeziums, corresponding to monomials of higher degree.
Each trapezium corresponds to a polynomial $Q_i^k$ computed by 
an approximate $k$th root as explained in Section~\ref{sec:root}.
It enables to lower the degree of $P$, except for monomials whose degree in $x$ is 
less than $k^2$ that will be treated at the end. We iterate this process until we get
a polynomial of degree less than $\frac{d}{k}$ (right figure) 
to which we will apply Corollary \ref{cor:vandermonde}.

\medskip

\textbf{Notation.} We will denote $\lceil x \rceil_k = k \left\lceil \frac x k \right\rceil$ the least integer
larger or equal to $x$ and divisible by $k$.

\medskip

\textbf{First step: lower the degree.}
Set $d = \deg P$, $m_0 = \lceil d \rceil_k$
and $P_0 := P$.
We apply Proposition \ref{prop:root} to $P_0=P$, with $P_0$ considered as a polynomial 
of total degree $\leq m_0$ and $m=m_0$, $n=0$.
It yields a polynomial $Q_0(x,y)$ such that $\deg_x (P+x^{m_0}-Q_0^k) < m_0 - \frac{m_0}{k}$.
That is to say we have canceled a trapezium, which is there the triangle $(m_0,0)$, $(m_0- \frac{m_0}{k},0)$,
$(m_0- \frac{m_0}{k},\frac{m_0}{k})$.

We then set $m_1 = \lceil m_0 \rceil_k - \frac{\lceil m_0 \rceil_k}{k}$
and $P_1 = P_0 + x^{m_0} - Q_0^k$. Note that $\deg_x P_1 < m_1$ and we apply Proposition \ref{prop:root} to $P_1$.

To iterate the process, consider the decomposition $P_i = P_i'+ x^{m_i}\cdot P''_i$
with $\deg_x P_i' < m_i$.
We apply Proposition \ref{prop:root} to $P_i'$ (with $m=\lceil m_i \rceil_k$ and $n=n_i$ such that
$\lceil m_i \rceil_k+n_i=m_0$) that yields $Q_i$ such that $P_i'+x^{\lceil m_i \rceil_k}y^{n_i}-Q_i^k$ has no monomials 
in the corresponding trapezium whose $x$-coordinates are in between $\lceil m_i \rceil_k$ and 
$m_{i+1} := \lceil m_i \rceil_k - \frac{\lceil m_i \rceil_k}{k}$.
Notice that $P_{i+1} := P_i'+x^{\lceil m_i \rceil_k}y^{n_i}-Q_i^k + x^{m_i}\cdot P''_i$ also does not have monomials
in this trapezium.

Here is an example, set $d = 45$ and $k=3$ then we get 
$m_0 = 45$, $m_1 = 30$, $m_2=20$, $m_3=14$, $m_4 = 10$, $m_5=8$ and then we stop since 
$m_5 < k^2$. It implies that the first trapezium has its $x$-coordinates in between $45$ and $30$,
the second one between $30$ and $20$,... The height of the left side of each trapezium is always $\frac d k = 15$.
The picture is the following:
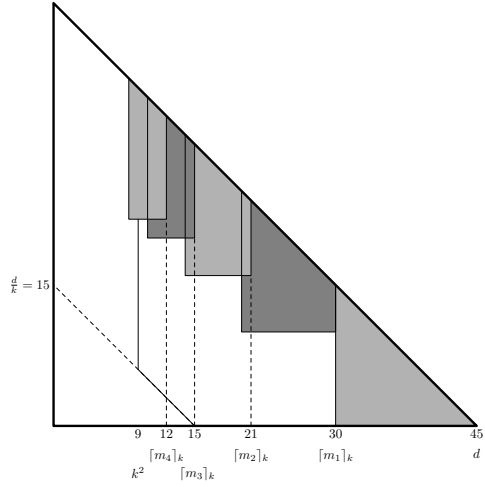
\begin{center}
\begin{figure}[H]
  \includegraphics[scale=0.5]{fig_waring.7}
\end{figure}
\end{center}


\textbf{End of iterations.}
We iterate the process until we reach monomials whose degree in $x$ is less than $k^2$.
That is to say we look for $\ell$ such that $m_\ell \le k^2$.

First notice that
\begin{align*}
  m_{i+1} &=  \lceil m_i \rceil_k - \frac{\lceil m_i \rceil_k}{k} \\
          &= (k-1)\left\lceil\frac {m_i}{k} \right\rceil  \\
          &\le \left(1-\frac 1k\right) m_i + k -1. \\
\end{align*}
Then, by induction
\begin{align*}
  m_i &\le  \left(1-\frac 1k\right)^i m_0 + (k-1)\left(1+\left(1-\frac 1k\right)+\left(1-\frac 1k\right)^2+\cdots+\left(1-\frac 1k\right)^{i-1}\right) \\
      &\le \left(1-\frac 1k\right)^i m_0 + k(k-1) \\
      &\le  (d+k)e^{-\frac i k} + k(k-1), \qquad \text{ since } \left(1-\frac 1k\right) \le e^{-\frac 1k}. \\
\end{align*}

Now, for $\ell \ge k \ln(\frac{d}{k}+1)$ we get $m_\ell \le k^2$.

\medskip

\textbf{Fall of the total degree.}
At the end of the first series of iterations the total degree 
(of the monomials whose degree in $x$ is more or equal to $k^2$)
falls (see the picture below).
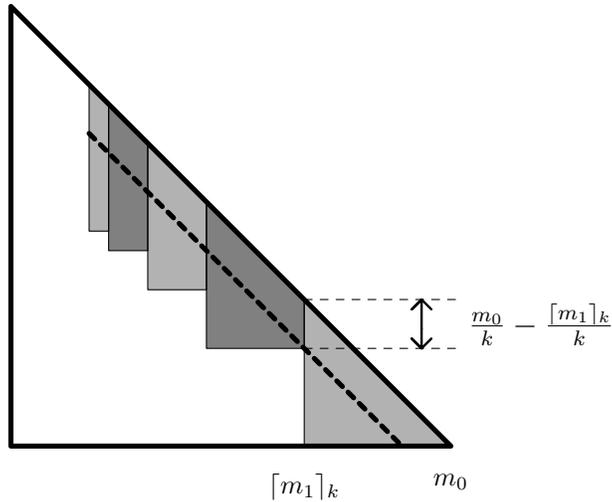
\begin{center}
\begin{figure}[H]
  \includegraphics[scale=1]{fig_waring.9}
\end{figure}
\end{center}
We give a lower bound for this fall $\delta_0$ of the degree (starting from degree $m_0$):
\begin{align*}
  \delta_0 &\ge \frac{m_0}{k} - \frac{\lceil m_1 \rceil_k}{k} \\
         &= \left\lceil \frac {d}{k} \right\rceil  - 
  \left\lceil \frac {k\left\lceil \frac {d}{k} \right\rceil - \left\lceil \frac {d}{k} \right\rceil}{k} \right\rceil 
  \quad (\text{since } d=\lceil m_0 \rceil_k) \\
         &\ge \left\lfloor \frac{\left\lceil \frac {d}{k} \right\rceil}{k} \right\rfloor \\
         &\ge \frac{d}{k^2} - 1. \\
\end{align*}

Therefore the total degree, starting now from degree $d$,
of the monomials whose degree in $x$ is more than $k^2$ has
fallen of more that $\delta \ge \frac{d}{k^2}-k$.

\medskip

\textbf{Iteration of the fall.}
Set $d_0 = d$. At each series of iterations the degree 
(of the monomials whose degree in $x$ is more or equal to $k^2$)
falls from $d_i$ to $d_{i+1} := d_i - \left\lfloor \frac{d_i}{k^2}-k \right\rfloor \le \left(1-\frac{1}{k^2}\right)d_i + k$,
so that (by a computation similar to the one for $m_i$ above) $d_i \le d e^{-\frac{i}{k^2}} + k^3$. 
Suppose that $d \ge 2k^4$, so that $\frac{d}{2k} + k^3 \le \frac{d}{k}$. 
Then for $\ell \ge k^2 \ln(2k)$, we get $d_\ell \le   \frac d k$.
Each fall of the degree needs less than 
$k\ln(\frac{d}{k}+1)$ iterations, so that 
we need to apply Proposition \ref{prop:root} many times, 
to get a total of $s_0=2k\ln(\frac{d}{k}+1) \times k^2 \ln(2k)$ 
$k$th powers.

\medskip

\textbf{Monomials of low degree in $x$.}
At this point, we have written 
$P = \sum_{i=1}^{s_0} Q_i^k  + P_1 + P_2$,
where $Q_1,\ldots,Q_{s_0}, P_1,P_2 \in F[x,y]$ are such that
$\deg Q_i^k \le \lceil d \rceil_k$, $\deg_x P_1 < k^2$, $\deg P_2 \le \frac{d}{k}$
(see the right picture below Theorem \ref{th:sum}).
By Corollary~\ref{cor:vandermonde} we can write $P_2$ as a sum 
$P_2= \sum_{i=1}^{s_2} Q_{i,2}^k$ of $s_2 \le k w_F(k)$ terms and 
$\deg Q_{i,2}^k \le k \left\lceil\frac{d}{k}\right\rceil = \lceil d \rceil_k < d+ k$.

Now write $P_1(x,y) = \sum_{0 \le j < k^2} x^j R_j(y)$, where $R_j \in F[y]$ with $\deg R_j \le d-j$.
By Corollary~\ref{cor:vandermonde}, write each $x^j$ as the sum of
$k w_F(k)$ terms of degree $\le j k$.
Then, for each $R_j(y)$ we apply the result in one variable \cite[Theorem 1.4 (iii)]{GaVa} (or we can do a similar work as before)
so that we can write (since $d \ge 2k^4$):
$R_j(y) = \sum_{i=1}^s S_{ij}^k(y)$ with $s \le k(w_F(k)+3\ln(k))+2$ and $\deg S_{ij}^k \le \deg R_j + k-1$.
We get $x^j R_j(y)$ as the sum of $s' \le k w_F(k) (k(w_F(k)+3\ln(k))+2)$, 
$k$th powers of degree $\le   j k + \deg R_j + k-1 \le d + k^3$ ($j=0,\ldots,k^2-1$).
Therefore, $P_1= \sum_{i=1}^{s_1} Q_{i,1}^k$ 
with $s_1 \le k^3 w_F(k) (k(w_F(k)+3\ln(k))+2)$ terms and 
$\deg Q_{i,1}^k \le d + k^3$.

\medskip

\textbf{Conclusion.}
For $d\ge 2k^4$ we can write $P(x,y)$ as the sum
$$P(x,y) = \sum_{i=1}^s Q_i^k(x,y)$$
such that $\deg Q_i^k \le d + k^3$ and $s\le s_0+s_2+s_1$ that is to say\footnote{This is the bound used to fill the numerical table of the introduction.}
$$s \le 2k^3\ln\left(\frac{d}{k}+1\right)\ln(2k) + k w_F(k) + k^3 w_F(k)(k(w_F(k)+3\ln(k))+2).$$
It yields the announced bound 
$s \le 2k^3\ln(\frac{d}{k}+1)\ln(2k) + 7k^4\ln(k) w_F(k)^2$.

\medskip

\textbf{Question.} Is it possible to have a sum 
$$P(x,y) = \sum_{i=1}^s Q_i^k(x,y)$$ such that $\deg Q_i^k \le \deg P + k^3$ and 
a bound $s$ depending only on $k$ and not on $\deg P$?



\begin{thebibliography}{99}

\bibitem{Bh} M. Bhaskaran, 
{\it Sums of mth powers in algebraic and abelian number fields.}
Arch. Math. (Basel) {17} (1966), 497-504; Correction,
ibid. 22 (1971), 370-371.

\bibitem{Bo} A. Bodin,
\emph{Decomposition of polynomials and approximate roots.}
Proc. Amer. Math. Soc.  138  (2010), 1989-1994. 

\bibitem{Ca1} M. Car, 
{\it New bounds on some parameters in the Waring problem for polynomials over a finite field.}
Contemporary Mathematics {461} (2008), 59-77. 

\bibitem{CG1} M. Car, L. Gallardo, 
{\it Sums of cubes of polynomials.} 
Acta Arith. {112} (2004), 41-50.

\bibitem{EfHa} G. Effinger, D. Hayes, 
{\it Additive number theory of polynomials over a finite field.} 
Oxford Mathematical Monographs, Clarendon Press, Oxford (1991).

\bibitem{Ga1} L. H. Gallardo, 
{\it On the restricted Waring problem over $\mathbf{F}_{2^n}[t]$.}
 Acta Arith. {42} (2000), 109-113.

\bibitem{GaVa} L.~Gallardo, L.~Vaserstein,
\emph{The strict Waring problem for polynomial rings.}
J. Number Theory  128  (2008), 2963--2972.

\bibitem{Ku} R.M. Kubota, 
\emph{Waring's problem for $\mathbf{F}_q[x]$.} 
Dissertationes Math. (Rozprawy Mat.) {117} (1974).

\bibitem{Pa} R.E.A.C Paley, 
{\it Theorems on polynomials in a Galois field.}
Quarterly J. of Math. {4} (1933), 52-63.

\bibitem{Va1} L.N. Vaserstein, 
{\it Waring's problem for algebras over fields.}
J. Number Theory {26} (1987), 286-298.

\bibitem{Va3} L.N. Vaserstein, 
{\it Sums of cubes in polynomial rings.}
Math. Comput. {193} (1991), 349-357.

\bibitem{Va4} L.N. Vaserstein, 
{\it Ramsey's theorem and Waring's problem.}
In {\it The Arithmetic of Function Fields} (eds D. Goss and al), de Gruyter, 
NewYork-Berlin, (1992).

\bibitem{We} W.A. Webb, 
{\it Waring's problem in $GF[q,x]$.}  
Acta Arith. {22} (1973), 207-220.

\bibitem{LiWo} Y.-R. Yu, T. Wooley, 
{\it The unrestricted variant of Waring's problem in function fields.}
Funct. Approx. Comment. Math. {37} (2007), 285-291. 

\end{thebibliography}
\end{document}